\documentclass[12pt, reqno]{amsart}
\usepackage{amsfonts,amssymb,graphicx,enumerate,youngtab}
\usepackage{fourier}

\numberwithin{equation}{section}

\theoremstyle{definition}

\theoremstyle{plain}
\newtheorem{theorem}{Theorem}[section]

\newtheorem{corollary}[theorem]{Corollary}
\newtheorem{lemma}[theorem]{Lemma}

\newtheorem{conjecture}[theorem]{Conjecture}
\newtheorem{question}[theorem]{Question}

\newcommand{\CC}{\mathbf{C}}

\newcommand{\ZZ}{\mathbf{Z}}

\newcommand{\OO}{\mathcal{O}}

\newcommand{\hh}{\mathfrak{h}}
\newcommand{\triv}{\mathrm{triv}}

\newcommand{\la}{\langle}
\newcommand{\ra}{\rangle}

\newcommand{\ttt}{\mathfrak{t}}

\begin{document}

\title{Subspace arrangements and Cherednik algebras}

\author{Stephen Griffeth}

\address{Instituto de Matem\'atica y F\'isica \\
Universidad de Talca \\ sgriffeth@inst-mat.utalca.cl}

\begin{abstract}
The purpose of this article is to study the relationship between numerical invariants of certain subspace arrangements coming from reflection groups and numerical invariants arising in the representation theory of Cherednik algebras. For instance, we observe that knowledge of the equivariant graded Betti numbers (in the sense of commutative algebra) of any irreducible representation in category $\OO$ is equivalent to knowledge of the Kazhdan-Lusztig character of the irreducible object (we use this observation in joint work with Fishel-Manosalva). We then explore the extent to which Cherednik algebra techniques may be applied to ideals of linear subspace arrangements: we determine when the radical of the polynomial representation of the Cherednik algebra is a radical ideal, and, for the cyclotomic rational Cherednik algebra, determine the socle of the polynomial representation and characterize when it is a radical ideal. The subspace arrangements that arise include various generalizations of the $k$-equals arrangment.  In the case of the radical, we apply our results with Juteau together with an idea of Etingof-Gorsky-Losev to observe that the quotient is Cohen-Macaulay for positive choices of parameters. In the case of the socle (in cyclotomic type), we give an explicit vector space basis in terms of certain specializations of non-symmetric Jack polynomials, which in particular determines its minimal generators and Hilbert series and answers a question posed by Feigin and Shramov.
\end{abstract}

\thanks{I thank Misha Feigin, Steven Sam and Jos\'e Simental for very helpful comments on a preliminary version of this paper and Carlos Ajila for pointing out a number of corrections to the first arxiv version. I acknowledge the financial support of Fondecyt Proyecto Regular 1190597.}

\maketitle

\section{Introduction}

\subsection{Overview} Given a graded module over a polynomial ring, there are a number of invariants encoding its structure and complexity. In roughly increasing order of subtlety and computational inaccessibility, some of these are the degrees of its minimal generators, its Hilbert polynomial, its Hilbert series, and its Betti table. In general, these invariants are difficult to calculate explicitly, but for certain classes of modules with more structure one might hope for closed formulas or combinatorial algorithms. Together with Berkesch-Zamaere and Sam \cite{BGS}, we observed that one (type of) module for which more can be said is the ideal of the $k$-equals arrangement in $n$-space. The key point is that this ideal is a unitary irreducible object of category $\OO_c$ for the rational Cherednik algebra of the symmetric group, so that technology developed for representation theory may be applied to study it. In this paper we develop a general framework for handling this interaction between Cherednik algebras and commutative algebra, and attempt to answer the natural (if somewhat vague) question: which graded modules may be studied using Cherednik algebras? As it turns out, the class of ideals to which Cherednik algebra techniques may be fruitfully applied is much broader than we realized in \cite{BGS}, in two different senses: firstly, we can potentially apply representation theoretic tools to compute the graded equivariant Betti numbers of any irreducible object of category $\mathcal{O}_c$, unitary or not, and secondly, it is not necessary to actually construct a BGG-type resolution to do so. In the remainder of this introduction we will make these observations more precise, and state our main theorems. We remark that the two streams of ideas coming together here have been studied previously, in commutative algebra in \cite{Sid}, and from the Cherednik algebra point of view in \cite{FeSh}.

\subsection{Subspace arrangements and their ideals}  Here, to orient the reader unfamiliar with Cherednik algebras and their representations, we give a class of examples of the type of ideal to which our methods apply, and state three of our main theorems: we give a set of minimal generators and a combinatorial formula for the coefficients of the Hilbert series, and prove that a certain class of subspace arrangements is Cohen-Macaulay. Although our proof of the theorems we are about to state requires a good deal of the machinery developed to study the representation theory of Cherednik algebras, the theorems themselves may be stated without any of the representation theoretic language (though the statement of the second theorem requires certain orthogonal polynomials). Fix integers $k,\ell,m,$ and $n$ such that $\ell,n \geq 1$, $2 \leq k \leq n$ and $0 \leq m \leq k-1$. Define the sets
$$X_{k,\ell,n}=\{(x_1,\dots,x_n) \in \CC^n \ | \ \hbox{ $\exists S \subseteq \{1,2,\dots,n\}$ with $|S|=k$ and $x_i^\ell=x_j^\ell$  $\forall i,j \in S$} \}$$ and
$$Y_{m+1,n}=\{(x_1,\dots,x_n) \in \CC^n \ | \ \hbox{ $\exists S \subseteq \{1,2,\dots,n\}$ with $|S|=m+1$ and $x_i=0$  $\forall i \in S$} \}.$$ Let $I=I_{k,\ell,m,n} \subseteq \CC[x_1,\dots,x_n]$ be the ideal of polynomial functions vanishing on the union $X_{k,\ell,n} \cup Y_{m+1,n}$.  Notice that if $m=k-1$ then $Y_{m+1,n} \subseteq X_{k,\ell,n}$ and hence $I$ is simply the ideal of $X_{k,\ell,n}$. 

First we will give an explicit set of generators for $I$. To do so, we require certain partitions, which we will visualize as Young diagrams. Firstly, divide $n$ by $k-1$ to obtain a quotient $q$ and remainder $r$ with $0 \leq r < k-1$. Let $((k-1)^q,r)$ be the partition recording the result of this division. For example, if $n=8$ and $k=4$ then the Young diagram of this partition is
$$\yng(3,3,2).$$ Now we split this into two partitions $\lambda^0$ and $\lambda^1$ as follows: if $m=k-1$ take $\lambda^0$ to be the partition itself and $\lambda^1=\emptyset$. Otherwise take $\lambda^1$ to be the partition whose shape consists of all boxes (weakly) down and to the right of the $m+1$th box along the top row, and take $\lambda^0$ to be the remaining diagram. Thus for example, with $k=4$ and $n=8$ as above, if $m=1$ then the pair $(\lambda^0,\lambda^1)$ is
$$\yng(1,1,1) \qquad \yng(2,2,1).$$ We define a \emph{standard Young tableau} on the pair $(\lambda^0,\lambda^1)$ to be a bijection from their boxes to the set of integers $\{1,2,\dots,n\}$ that is increasing from top to bottom and left to right (within each $\lambda^i$). An example is
$$\young(1,4,5) \qquad \young(23,67,8).$$ Finally, given a standard Young tableau $T$, we define a polynomial $f_T$ to be the product
$$f_T=\prod_{b \in \lambda^1} x_{T(b)} \prod_{ \substack{b, b' \in \lambda^\cdot \\ b \ \text{above} \ b' \ \text{in the same column}}} (x_{T(b)}^\ell-x_{T(b')}^\ell).$$ So for our example $T$ above,
$$f_T=x_2x_3x_6x_7 x_8 (x_1^\ell-x_4^\ell)(x_1^\ell-x_5^\ell)(x_4^\ell-x_5^\ell)(x_2^\ell-x_6^\ell)(x_2^\ell-x_8^\ell)(x_6^\ell-x_8^\ell)(x_3^\ell-x_7^\ell).$$

We refer to the polynomial $f_T$ as the \emph{Garnir polynomial} associated to $T$. Their representation-theoretic significance is that, for $\ell>1$ or $m=k-1$, their span is the lowest occurrence of the $G(\ell,1,n)$-representation indexed by $(\lambda^0,\emptyset,\dots,\emptyset,\lambda^1)$. 

\begin{theorem} \label{generation thm}
The ideal $I$ is generated by the set of Garnir polynomials $f_T$ for $T$ ranging over all standard Young tableaux on $(\lambda^0,\lambda^1)$. Moreover, if $\ell >1$ or $m=k-1$ this is a minimal generating set. 
\end{theorem}

The theorem answers a question posed by Feigin and Shramov in \cite{FeSh}: in their Remark 2.21 the equalities hold (in fact, this paper started as an attempt to answer a number of questions suggested by their work and that of Feigin \cite{Fei}, all having to do with unitarity and deciding when certain reprentations of the Cherednik algebra arising from the ring of polynomial functions are irreducible). We mention two previously known consequences of this theorem: when $m=0$ and $\ell=1$ Li and Li \cite{LiLi} obtained this result (and a new proof of a graph-theoretic result of Tur\'an \cite{Tur}), and Sidman \cite{Sid} generalized their result to $m=0$ and $\ell$ arbitrary (see also \cite{ScSi} for a survey of results on subspace arrangements). In fact, both of these papers work first with much larger sets of generators, but it is not difficult to find the minimal generating set given their work (\cite{LiLi} does this explicitly). We note that de Loera \cite{deL} has conjectured (in the case $\ell=1$ and $m=0$) that these much larger sets of generators are actually Groebner bases. The generators in the conjectural Groebner basis are indexed by set partitions of $n$ with $k-1$ blocks; there are many more of these than there are standard Young tableaux on the shape $\lambda$. For instance, when $n=12$ and $k=5$  there are $462$ standard Young tableaux of shape $\lambda$, but $611,501$ set partitions of $12$ with $4$ blocks. On the other hand, already for $n=4$, $m=0$, $\ell=1$ and $k=3$, the minimal generating set appearing in our theorem is not a Groebner basis with respect to the lex order.

Our proof of the theorem is of a completely different flavor to those of Li-Li and Sidman: it is a consequence of our study of the structure of the ring of polynomials as a module over the rational Cherednik algebra for the complex reflection group $G(\ell,1,n)$. The main point, parallel to the case $\ell=1$, is that $I$ may be realized as the socle of the polynomial representation of the rational Cherednik algebra, and in fact, it may be realized as a unitary representation (our paper \cite{Gri3} contains a classification of unitary representations in category $\OO_c$ of cyclotomic rational Cherednik algebras). 

Our next theorem involves the non-symmetric Jack polynomials for the cyclotomic reflection group $G(\ell,1,n)$; for the precise definition we refer to Section \ref{cyclotomic}. We write $f_\mu$ for the non-symmetric Jack polynomial indexed by a composition $\mu \in \ZZ_{\geq 0}^n$ with the parameter $c_0$ specialized to $1/k$. If $\ell=1$ then $f_\mu$ is the classical non-symmetric Jack polynomial, and for general $\ell$ these may be expressed rather simply in terms of the $\ell=1$ version (see \cite{DuOp}). The composition $\mu$ is \emph{$(j,k)$-admissible} if for all $k \leq i \leq n$ we have $\mu^-_i \geq \mu^-_{i-(k-1)}+j$ with equality implying $$w_\mu^{-1}(i)<w_\mu^{-1}(i-(k-1)),$$ where $\mu^-$ is the non-decreasing re-arrangement of $\mu$ and $w_\mu \in S_n$ is the longest permutation with $w_\mu \cdot \mu=\mu^-$. The next theorem is a non-symmetric, cyclotomic analog of a result of Feigin-Jimbo-Miwa-Mukhin \cite{FJMM}. Its proof will be given in Section \ref{cyclotomic} after a detailed study of the socle of the polynomial representation of the rational Cherednik algebra.

\begin{theorem} \label{basis thm}
The ideal $I$ has the following vector space basis:
$$I=\CC \{f_\mu \ \vert \ \mu_{m+1}^- \geq 1 \ \hbox{and $\mu$ is $(\ell,k)$-admissible} \}.$$
\end{theorem}

By noting that the polynomial degree of $f_\mu$ is $\mu_1+\mu_2+\cdots+\mu_n$ this theorem in particular gives combinatorial formulas for the coefficients of the Hilbert series of $I$ (that is, for the dimensions of the homogeneous pieces of $I$). In fact, our proof of Theorem \ref{generation thm} goes via first proving Theorem \ref{basis thm}. Along the way we observe that (for $\ell >1$ or $m=k-1$) $I$ may be realized as a unitary representation of the Cherednik algebra (we then deduce the case $\ell=1$ from the case $\ell=2$). We remark that the lowest degree Jacks appearing here \emph{do} always give a minimal generating set for $I$, even in case $\ell=1$. 

We will further leverage unitarity of $I$ to get more detailed information: in the sequel \cite{FGM} to this paper we will give a non-negative combinatorial formula for the Betti numbers of $I$ in terms of Littlewood-Richardson numbers. For the case $\ell=1$ ($W=S_n$), Bowman, Norton and Simental \cite{BNS} have another approach to the calculation of these via a direct construction of the BGG resolution that was conjectured in \cite{BGS} (and together with Norton, in \cite{GrNo} we gave yet another method for producing resolutions by standard objects, but for a class of examples different from the unitary representations). The BNS approach also applies to some unitary modules for $\ell > 1$, but the conditions they impose are too strict to include, for instance, the ideal $I$ for most values of $k$, $\ell$ and $m$. 

Our next theorem concerns Cohen-Macaulayness of a certain subspace arrangement. Fix positive integers $n$, $\ell$ and $k$ with $\ell>1$ (this hypothesis could be dropped, but at the cost of more notation, and in any case is already covered by the results of \cite{EGL}), $1 \leq k \leq n$, let $m$ be an integer $0 \leq m \leq n$, and let $p=\lfloor (n-m)/k \rfloor$ be the greatest integer which is at most $(n-m)/k$. Consider the set $Z_{\ell,m,k}$ consisting of points $x=(x_1,\dots,x_n) \in \CC^n$ such that there exist disjoint subsets $S_1,\dots,S_p,T \subseteq \{1,2,\dots,n\}$ (which may depend on $x$) satisfying
\begin{itemize} 
\item[(a)] $|T|=m$ and $|S_r|=k$ for $1\leq r \leq p$,
\item[(b)] $x_i=0$ if $i \in T$, and
\item[(c)] $x_i^\ell=x_j^\ell$ if $i,j \in S_r$ for some $1 \leq r \leq p$. 
\end{itemize} Our result is:
\begin{theorem} \label{CM theorem}
The variety $Z_{k,\ell,m}$ is Cohen-Macaulay.
\end{theorem} In case $\ell=1$ and $m=0$ this is proved in \cite{EGL}, Proposition 3.9. We prove the theorem by using a slight variation of an idea of \cite{EGL}, that minimally supported representations of the rational Cherednik algebra are Cohen-Macaulay, together with our results \cite{GrJu} to prove that for a certain choice of deformation parameter the coordinate ring of $Z_{k,\ell,m}$ is a minimally supported representation in the principal block (noting that $Z_{k,\ell,m}$ is the closure of the set $O$ appearing in Corollary \ref{classical radical}). To state the rest of our main results we require the machinery of rational Cherednik algebras.

\subsection{Dunkl operators and Cherednik algebras} In order to make the remainder of this introduction as self-contained as possible, before stating the rest of our main theorems we must introduce some of the notation and terminology from the theory of rational Cherednik algebras. The starting point is \emph{Dunkl operators}. 

Dunkl operators were first introduced by Dunkl \cite{Dun}, whose aim was to deform harmonic analysis on the unit sphere. Each family of Dunkl operators arises from a finite group $W$ of linear transformations of a vector space, which in Dunkl's original theory was required to be a real reflection group; Dunkl and Opdam \cite{DuOp} later generalized the construction to all complex reflection groups. 

The Dunkl operators are a family of commuting operators on polynomial functions in several variables, deforming the usual partial derivatives. Given a finite linear group $W \subseteq \mathrm{GL}(\hh)$, where $\hh$ is a finite dimensional $\CC$-vector space, and a vector $v \in \hh$, the formula for the Dunkl operator $D_v$ is
$$D_v(f)=\partial_v(f)-\sum_{r \in R} c_r \alpha_r(v) \frac{f-r(f)}{\alpha_r} \quad \hbox{for $f \in \CC[\hh]$,}$$ where $\CC[\hh]$ is the ring of polynomial functions on $\hh$, $R$ is the set of reflections in $W$, and for each $r \in R$ we have fixed a number $c_r \in \CC$ and a linear form $\alpha_r \in \hh^*$ with zero set equal to the fixed space of $r$. The numbers $c_r$ are required to be constant on $W$-conjugacy classes, $c_{w r w^{-1}}=c_r$ for all $w \in W$ and $r \in R$. We will write $c=(c_r)_{r \in R}$ for the collection of all of them, and call $c$ the \emph{deformation parameter}. In this paper we will further assume that the $c_r$'s satisfy the reality condition $c_{r^{-1}}=\overline{c_r},$ where the over-bar indicates complex conjugation, whenever we must ensure that the Fourier transform and contravariant form are well-defined (this is no real loss of generality). 

The \emph{rational Cherednik algebra} is the algebra $H_c$ of operators on $\CC[\hh]$ generated by the Dunkl operators $D_v$ for $v \in \hh$, the algebra $\CC[\hh]$ acting on itself by multiplication, and the group $W$. We define \emph{category $\OO_c$} to be the category of finitely-generated $H_c$ modules on which each Dunkl operator act locally nilpotently. Thus, for example, $\CC[\hh]$ belongs to $\OO_c$. Each object of $\OO_c$ is a finitely generated graded $\CC[\hh]$ module. The grading involves certain numerical invariants: the $c$-function $c_E$ is defined to be the function which assigns the scalar by which the central element
$$z_c=\sum_{r \in R} c_r (1-r)$$ acts on the irreducible representation $E$ of $\CC W$. Thus in particular for the trivial representation $E$ we have $c_E=0$. We will write $M_d$ for the degree $d$ subspace of $M$.

The most accessible objects of $\OO_c$ are the \emph{standard modules}. Given an irreducible representation $E$ of $\CC W$ the corresponding standard module for $H_c$ is $\Delta_c(E)=\CC[\hh] \otimes E$, which may be constructed via parabolic induction. Namely, the \emph{PBW theorem} for the rational Cherednik algebra states that multiplication induces an isomorphism of vector spaces 
$$\CC[\hh] \otimes \CC W \otimes \CC[\hh^*] \cong H_c,$$ and hence
$$\Delta_c(E)=\mathrm{Ind}_{\CC[\hh^*] \rtimes W}^{H_c}(E) \cong \CC[\hh] \otimes E.$$ In particular, the defining representation $\CC[\hh]$ of $H_c$ is isomorphic to the standard module $\Delta_c(\triv) \cong \CC[\hh]$. The Cherednik algebra possesses an anti-automorphism (the \emph{Fourier transform}) which interchanges $\CC[\hh]$ and $\CC[\hh^*]$, and using it and the argument familiar from Lie theory produces a \emph{contravariant form} $(\cdot,\cdot)_c$ on $\Delta_c(E)$. The quotient of $\Delta_c(E)$ by the radical of this form is the irreducible head $L_c(E)$ of $\Delta_c(E)$, and we say that $L_c(E)$ is \emph{unitary} if the induced form is positive definite. 

Given an $H_c$-module $M$ we define $H_0(\hh^*,M)=\CC \otimes_{\CC[\hh]} M$, where $\CC$ is a $\CC[\hh]$-module via evaluation at the origin $0 \in \hh$. This defines a right-exact functor from $\OO_c$ to graded $\CC W$-modules, and we write $H_i(\hh^*,M)$ for the homology of its left derived functor. If we regard $M$ as a graded module over $\CC[\hh]$, then $H_i(\hh^*,M)=\mathrm{Tor}_i(M,\CC)$ is of graded dimension $\sum \beta_{i,j} t^j$, so we recover the graded Betti numbers from knowledge of the graded vector space $H_i(\hh^*,M)$. Similarly, we define
$$H^0(\hh,M)=\{m \in M \ | \ D_v(m)=0 \ \hbox{for all $v \in \hh$} \}.$$ This is a left-exact functor from $\OO_c$ to graded $\CC W$-modules, and we write $H^i(\hh,M)$ for the cohomology of its right derived functor. We note that there is an isomorphism of functors $\mathrm{Hom}_{H_c}(\Delta_c(E),\cdot) \cong \mathrm{Hom}_{\CC W}(E,H^0(\cdot))$, allowing us to recover the higher Ext's from standard objects from knowledge of $H^i(\hh,\cdot)$. 

The following theorem is a synthesis of Theorem \ref{duality theorem}, the material from subsection \ref{Tor and Ext}, Corollary \ref{purity}, and results of Huang-Wong \cite{HuWo} and Ciubotaru \cite{Ciu} that we have reproduced here in Section \ref{homology}. 

\begin{theorem} \label{Hodge decomposition}
Let $L$ be an irreducible object of $\OO_c$. Then there exists an isomorphism of graded $W$-modules:
$$H_i(\hh^*,L) \cong H^i(\hh,L)$$ and consequently an isomorphism of graded $W$-modules
$$\mathrm{Tor}_i(L,\CC) \cong \bigoplus_{E \in \mathrm{Irr}(\CC W)} \mathrm{Ext}^i_{\OO_c}(\Delta_c(E),L) \otimes_\CC E.$$ Moreover, for any object $M$ of $\OO_c$ the following \emph{equivariant purity property} holds: if the $E$-isotypic component of the degree $d$ piece of the Tor-group $\mathrm{Tor}_i(M,\CC)_{E,d} \neq 0$ is not zero, then $d=c_E$. Finally,  if $L$ is a unitary irreducible object of $\OO_c$, then
$$\mathrm{Ext}^i(\Delta_c(E),L) \cong \mathrm{Hom}_{\CC W} (E,L_{c_E-i} \otimes \Lambda^i \hh^*).$$ 
\end{theorem}

For the first assertion of the theorem, the fact that irreducible objects are self-dual is crucial. We remark that the assertion about unitary irreducibles in the theorem is essentially a result of \cite{Ciu} and \cite{HuWo}; it should be regarded as an analog of the Hodge decomposition theorem for compact manifolds. So our new contributions are the observations that for irreducible objects of $\OO_c$ the Tor-groups and the $\hh$-cohomology are isomorphic, and that the grading on the Tor-groups satisfies the equivariant purity property, allowing us to recover the \emph{graded} equivariant Betti table of $L$ from the Kazhdan-Lusztig character formula.
 
Our next results are our attempt to make precise the generality in which the Cherednik algebra technology (and in particular the preceding theorem) may be applied to obtain information about the ideals of subspace arrangements. There are two situations in which this is straightforward: first, if the \emph{coordinate ring} of the subspace arrangement is an irreducible object of $\OO_c$, and second if the \emph{ideal}  of the subspace arrangement is an irreducible object of $\OO_c$. Our next two main theorems deal with these two possibilities. We note that if the coordinate ring of a subspace arrangement is an irreducible object of $\OO_c$, then since it is generated by $1$ as a $\CC[\hh]$-module it must be equal to $L_c(\triv)$ and hence its support (the subspace arrangement) must be a single $W$-orbit of strata whose ideal is the radical of $\Delta_c(\triv)$ (on the other hand, when the ideal of a subspace arrangement is an irreducible object, it is not in general so easy to identify its lowest weight). For the terminology involving Schur elements of complex reflection groups and positive hyperplanes, we refer to \cite{GrJu}.

\begin{theorem} \label{radical thm}
Let $O$ be a single $W$-orbit of strata of the reflection arrangement of $W$. Suppose that for some $p \in O$, $W_p=W_1 \times \cdots \times W_p$ is the product decomposition of the stabilizer $W_p$ of $p$ into irreducible reflection subgroups, and write $\hh^*_i$ for the dual reflection representation of $W_i$. Then $\mathrm{Rad}(\Delta_c(\triv))=I(\overline{O})$ if and only if
\begin{enumerate}
\item[(a)] $c_{\hh^*_i}=1$ for all $i$, and
\item[(b)] The parameter $c$ does not belong to any positive hyperplane $H$ such that the quotient $s/s'$ of the principal Schur element $s$ for $W$ by the principal Schur element $s'$ for $W_p$ vanishes on the intersection of $H$ with the positive cone. (See \cite{GrJu} for the notation and an explicit description of the relevant Schur elements).
\end{enumerate}
\end{theorem}

We note that condition (a) of the theorem above has appeared already in the work \cite{Fei} of Feigin, who classified when the ideals of subspace arrangements are stable by the Dunkl operators (in fact, the theorem above is an immediate consequence of his work in combination with our theorem with Juteau). Aside from its application to the calculation of Betti numbers of subspace arrangements arising from reflection groups, our interest in this theorem is that it gives an explicit construction of the module $L_c(\triv)$ in a number of interesting cases.

For the groups in the infinite family $G(\ell,1,n)$ the theorem may be made much more explicit. The deformation parameter in this case is a tuple $c=(c_0,d_0,\dots,d_{\ell-1})$ (see Section \ref{cyclotomic} for this).

\begin{corollary} \label{classical radical}
Suppose $W=G(\ell,1,n)$. Let $O$ be the $W$-orbit of strata corresponding to the parabolic subgroup $(S_k )^{\times p} \times G(\ell,1,m)$. Then the radical of $\Delta_c(\triv)$ is $I(O)$ if and only if
\begin{enumerate}
\item[(a)] $c_0=1/k$, 
\item[(b)] $d_0-d_{\ell-1}+\ell (m-1)c_0=1$, 
\item[(c)] $p=\lfloor (n-m)/k \rfloor$, and
\item[(d)] no equation of the form $d_0-d_{-j}+\ell m' c_0=j$ holds, for $j$ a positive integer not congruent to $0$ mod $\ell$ and $m \leq m' \leq n-1$. 
\end{enumerate}
\end{corollary}

We note that condition (c) determines $p$ given $k$ and $m$, independent of the parameters, which implies that only for relatively large $m$, $p$ can $I(O)$ be equal to the radical. This corollary, combined with Theorem \ref{Hodge decomposition} and the result of Rouquier-Shan-Varagnolo-Vasserot \cite{RSVV} expresses the Betti numbers of the ideal of $O$ in terms of parabolic affine Kazhdan-Lusztig polynomials (though in fact we expect a simpler closed formula should be within reach, we do not address that here).

The other possibility is that the ideal of the subspace arrangement is an irreducible object of $\OO_c$, or in other words, equal to the socle of the polynomial representation. We are able to achieve a characterization of this situation only for the groups in the infinite family $G(\ell,1,n)$; for these groups we can in fact give a much more detailed result giving a vector space basis of the socle of the polynomial representaiton for all $c$ such that $c_0 \notin \ZZ_{>0}$ (the non-symmetric Jack polynomials involved have poles at such $c_0$). The result (see Theorem \ref{socle thm}) is rather technical, so here we state only the characterization of when the socle is the ideal of a subspace arrangement. 

\begin{theorem} \label{radical socle}
Suppose $W=G(\ell,1,n)$. Then the socle of the polynomial representation is a radical ideal if and only if one of the following mutually exclusive possibilities occurs:
\begin{enumerate}
\item[(a)] It is equal to $\CC[\hh]$,
\item[(b)] there exists $2 \leq k \leq n$ such that it is the ideal of $X_{k,\ell,n}$,
\item[(c)] there exists $0 \leq m \leq n-1$ such that it is the ideal of $Y_{m+1,n}$, or
\item[(d)] there exist $2 \leq k \leq n$ and $0 \leq m \leq k-2$ such that it is the ideal of $X_{k,n} \cup Y_{m+1,n}$.  
\end{enumerate} The necessary and sufficient conditions on the deformation parameter $c$ for each case to occur are as follows:
\begin{enumerate}
\item[(a)] The parameter $c_0$ is not of the form $j/k$ for positive coprime integers $j$ and $k$ with $k \geq 2$, and no equation of the form $$d_0-d_{-p}+\ell m c_o=p$$ for integers $p$ and $m$ with $0 \leq m \leq n-1$ and $p>0$ not divisible by $\ell$ holds.
\item[(b)] We have $c_0=1/k$ for some integer $k \geq 2$ and no equation of the form $$d_0-d_{-p}+\ell m c_0=p$$ for integers $m$ and $p$ with $0 \leq m \leq n-1$, $p$ not divisible by $\ell$, and
$$p  > x \ell \quad \hbox{if $x(k-1)+1 \leq m \leq (x+1)(k-1)$ for an integer $x$ holds.}$$
\item[(c)] We have $d_0-d_{\ell-1}+\ell m c_0=1$, the parameter $c_0$ is not of the form $j/k$ for positive coprime integers $j$ and $k$ with $k \geq 2$, and no equation of the form 
$$d_0-d_{-p}+\ell m' c_0=p$$ holds for integers $m'$ and $p$ such that either $0 \leq m' < m$ or $0 \leq m' \leq n-1$ and $p>1$.
\item[(d)] We have $c_0=1/k$ for an integer $k \geq 2$ and  $d_0-d_{\ell-1}+\ell m c_0=1$, and no equation of the form $$d_0-d_{-p}+\ell m' c_0=p$$ holds for integers $m'$ and $p$ such that $p$ is not divisible by $\ell$, $0 \leq m' \leq n-1$, and for each $x \in \ZZ_{\geq 0}$
$$p>x \ell \quad \text{if} \quad x(k-1) \leq m' < x(k-1)+m$$ and $$ p>x \ell +1 \quad \text{if} \quad x(k-1)+m \leq m' < (x+1)(k-1).$$
\end{enumerate} Moreover, in each case there are choices of parameters for which the socle is unitary.
\end{theorem}

We may paraphrase the last assertion of the theorem as follows: for the group $W=G(\ell,1,n)$, if the ideal of a subspace arrangement is a simple representation of the Cherednik algebra for \emph{some} value of the parameter $c$, then in fact $c$ may be chosen so that it is a \emph{unitary} irreducible representation. This tendency for the socle to be unitary when it is a radical ideal does not hold for all groups $W$. In particular, Feigin and Shramov \cite{FeSh} give an example where this fails for $W$ a Weyl group of type $D$. It would be very interesting to explore this phenomenon for the exceptional groups, but at the moment we do not have the necessary tools. Combining this theorem with Theorem \ref{Hodge decomposition} is the starting point of our joint work with Fishel and Manosalva \cite{FGM}, where we will obtain a non-negative combinatorial formula for the Betti numbers of the ideals of these subspace arrangements as a consequence of our study of unitary simple modules.

\subsection{Structure of the paper} In Section \ref{radical} we recall the basic definitions in the theory of rational Cherednik algebras, characterize when the radical of its polynomial representation is a radical ideal (proving Theorem \ref{radical thm}), and show that for positive choices of the parameters $L_c(\mathrm{triv})$ is minimally supported in its block, which when applied for the cyclotomic groups proves Theorem \ref{CM theorem}. In Section \ref{homology} we relate the graded Betti numbers from commutative algebra to graded multiplicities from representation theory and in particular prove Theorem \ref{Hodge decomposition}. Finally, in Section \ref{cyclotomic} we study the cyclotomic groups in more detail, prove Theorems \ref{generation thm}, \ref{basis thm}, and \ref{radical socle} and state a number of questions and conjectures on the structure of the polynomial representation. 

\section{The rational Cherednik algebra, and when the radical is radical} \label{radical}

\subsection{} Here we will give the basic definitions around Cherednik algebras and category $\OO_c$. We refer the reader to \cite{GGOR} for more details.

\subsection{Reflection groups} Let $\hh$ be a finite dimensional $\CC$-vector space and let $W \subseteq \mathrm{GL}(\hh)$ be a finite group of linear transformations of $\hh$. The set of \emph{reflections} in $W$ is
$$R=\{r \in W \ | \ \mathrm{codim}(\mathrm{fix}_\hh(r))=1 \}.$$ The linear group $W \subseteq \mathrm{GL}(\hh)$ is a \emph{reflection group} if it is generated by $R$. We emphasize that this notion depends not only on $W$ as an abstract group, but also on its representation as a linear group. 

\subsection{Dunkl operators} With the preceding notation, for each reflection $r \in R$ we fix a linear form $\alpha_r \in \hh^*$ with zero set equal to the fixed space of $r$,
$$\{v \in \hh \ | \ r(v)=v \}=\{v \in \hh \ | \ \alpha_r(v)=0\}.$$ The form $\alpha_r$ is well-defined up to multiplication by a non-zero scalar. We also choose $c_r \in \CC$ a complex number, subject to the requirements
$$c_{w r w^{-1}}=c_r \quad \text{and} \quad c_{r^{-1}}=\overline{c_r} \quad \hbox{for all $r \in R$ and $w \in W$.}$$ The second condition is not standard; we include it so that the Fourier transform and contravariant forms are well-defined.

We write $\CC[\hh]$ for the ring of polynomial functions on $\hh$. For a vector $v \in \hh$ and $f \in \CC[\hh]$ we will write $\partial_v(f)$ for the derivative of $f$ in the direction $\hh$, defined by
$$\partial_v(f)(x)=\lim_{h \to 0} \frac{f(x+hv)-f(x)}{h}.$$ The action of $W$ on $\hh$ induces an action on $\CC[\hh]$, and for $w \in W$ and $f \in \CC[\hh]$ we write $w(f)$ for the action of $w$ on $f$. Given $v \in \hh$, we define the \emph{Dunkl operator} $D_v$ by its action on $\CC[\hh]$:
$$D_v(f)=\partial_v(f)-\sum_{r \in R} c_r \la \alpha_r,v \ra \frac{f-r(f)}{\alpha_r}.$$

\subsection{The rational Cherednik algebra} The \emph{rational Cherednik algebra} is the subalgebra $H_c(W,\hh)$ of $\mathrm{End}_\CC(\CC[\hh])$ generated by 
\begin{enumerate}
\item[(a)] The group $W$,
\item[(b)] the algebra $\CC[\hh]$,
\item[(c)] for each $v \in \hh$, the Dunkl operator $D_v$. 
\end{enumerate}

\subsection{The Euler element}

The \emph{Euler field} is the differential operator
$$\mathrm{eu}=\sum_{i=1}^n x_i \frac{\partial}{\partial x_i },$$ where $x_1,\dots,x_n$ is a basis of $\hh^*$. It belongs to $H_c$: one checks
$$\mathrm{eu}=\sum x_i D_{v_i}+\sum_{r \in R} c_r (1-r)$$ for any choice of basis $v_1,\dots,v_n$ of $\hh$ with dual basis $x_1,\dots,x_n$ of $\hh^*$. 

\subsection{Category $\OO_c$} The category $\OO_c$ is the full subcategory of $H_c$-mod whose objects are finitely generated $H_c$ modules on which each Dunkl operator acts locally nilpotently. Given an irreducible representation $E$ of $\CC W$ the \emph{standard module} $\Delta_c(E)$ is defined by
$$\Delta_c(E)=\mathrm{Ind}_{\CC[\hh^*] \rtimes W}^{H_c}(E),$$ where $\CC[\hh^*] \rtimes W$ is the subalgebra of $H_c$ generated by the Dunkl operators and $W$, and where $D_v(e)=0$ for all $e \in E$ and $v \in \hh$. These are objects of $\OO_c$. As a $\CC[\hh] \rtimes W$-module, we have
$$\Delta_c(E)=\CC[\hh] \otimes E,$$ and the Euler element acts on the polynomial degree $d$ part of $\Delta_c(E)$ by the scalar $d+c_E$. The category $\OO_c$ is the Serre subcategory of $H_c$-mod generated by the standard modules. Hence each object $M$ of $\OO_c$ is $\CC$-graded with finite dimensional weight spaces,
$$M=\bigoplus_{d \in \CC} M_d, \quad \text{where} \quad M_d=\{m \in M \ | \ (\mathrm{eu}-d)^N \cdot m=0 \ \hbox{for $N$ sufficiently large} \}.$$ Moreover, $M_d=0$ unless $d \in c_E+\ZZ_{\geq 0}$ for some irreducible representation $E$ of $W$. 

\subsection{Duality} Fix a conjugate linear $W$-equivariant isomorphism $\hh \to \hh^*$, written $v \mapsto \overline{v}$, with inverse also written $x \mapsto \overline{x}$. There is a conjugate-linear anti-automorphism of $\CC W$ determined by $w \mapsto w^{-1}$ for $w \in W$.  These maps extend to a conjugate linear anti-automorphism of $H_c$ (here it is important that we have imposed the reality condition $c_{r^{-1}}=\overline{c_r}$ on the parameters), which we also write $h \mapsto \overline{h}$. If $W$ acts irreducibly on $\hh$, this is determined up to a non-zero scalar.

For a finite-dimensional $\CC$-vector space $V$ we write 
$$V^\vee=\{f: V \to \CC \ | \ f(av_1+bv_2)=\overline{a} f(v_1)+\overline{b} f(v_2) \ \hbox{for all $v_1,v_2 \in V$ and $a,b \in \CC$} \}$$ for the conjugate linear dual of $V$. This defines a contravariant functor from the category of finite-dimensional vector spaces to itself, which is conjugate-linear on Hom sets. 

Given an object $M$ of $\OO_c$, with grading $M=\bigoplus M_d$, we set
$$M^\vee=\bigoplus M_d^\vee.$$ We might regard $M^\vee$ as a \emph{restricted dual} of $M$ via the obvious embedding of $M^\vee$ in the full conjugate linear dual of $M$. We endow $M^\vee$ with the structure of an $H_c$-module by
$$h \cdot \phi(m)=\phi(\overline{h} \cdot m) \quad \hbox{for $h \in H_c$, $\phi \in M^\vee$, and $m \in M$.}$$ This defines a contravariant functor from $\OO_c$ to itself, which is conjugate linear on Hom sets.  

\subsection{The contravariant form}

Each standard module $\Delta_c(E)$ carries a \emph{contravariant form} $(\cdot,\cdot)_c$ defined by
$$(f_1 \otimes e_1,f_2 \otimes e_2)_c=(e_1,(\overline{f_1} \cdot f_2 \otimes e_2)(0)),$$ where $(\cdot,\cdot)$ is a fixed positive definite $W$-invariant Hermitian form on $E$ and for $m \in \Delta_c(E)$ we write $m(0)$ for the evaluation of $m$ at $0 \in \hh$ (which is an element of $E$). This form is Hermitian and its radical is the radical of $\Delta_c(E)$. The quotient $L_c(E)$ of $\Delta_c(E)$ is irreducible and carries an induced non-degenerate contravariant form, which we also write $(\cdot,\cdot)_c$. When this is positive definite, we call $L_c(E)$ \emph{unitary}. In any case, the contravariant form defines an isomorphism $L_c(E) \cong L_c(E)^\vee$. 

\subsection{Parabolic subgroups} A \emph{parabolic subgroup} of $W$ is a subgroup $W_v$ of the form
$$W_v=\{w \in W \ | \ w(v)=v \}$$ for some $v \in \hh$. Given a parabolic subgroup $W'$ of $W$, the \emph{stratum} of $\hh$ associated with $W'$ is
$$S_{W'}=\{v \in \hh \ | \ W_v=W' \}.$$ Conversely, given a stratum $S$, the parabolic subgroup of $W$ associated with $S$ is 
$$W_S=W_v \quad \hbox{for some (any) $v \in S$.}$$ The maps $S \mapsto W_S$ and $W' \mapsto S_{W'}$ are inverse bijections between the set of parabolic subgroups of $W$ and the set of strata in $\hh$. Moreover, the bijection is equivariant for the $W$-action on the set of strata induced by its action on $\hh$, and the $W$-action by conjugation on the set of parabolic subgroups. We will write $R_v=R \cap W_v$ for the set of reflection in $W_v$ (by a theorem of Steinberg, $W_v$ is generated by $R_v$ if $W$ is generated by $R$).

\subsection{Parabolic subgroups and Dunkl stable ideals} Let $v \in \hh$ and let $W_v$ be the corresponding parabolic subgroup. Write $$\hh=\hh^{W_v} \oplus U,$$ where $U$ is the (unique) $W_v$-stable complement to the fixed space $\hh^{W_v}$. The following lemma appeared in \cite{BGS}; we repeat it here for the convenience of the reader.
\begin{lemma} \label{submodule construction}
Let $I \subseteq \CC[U]$ be an ideal that is stable by the action of $H_c(W_v,U)$ and let $J$ be the ideal generated by $I$ in $\CC[\hh]=\CC[\hh^{W_v}] \otimes \CC[U]$. Then the ideal 
$$K=\bigcap_{w \in W} w(J)$$ is $H_c(W,\hh)$-stable.
\end{lemma}
\begin{proof}
By construction, $K$ is a $W$-stable ideal in $\CC[\hh]$. We need only verify that it is stable by the Dunkl operators. Let $f \in K$. By $W$-equivariance of the Dunkl operators, it suffices to verify that $D_y(f) \in J$ for all $y \in \hh$. We have
$$D_y(f)=\partial_y(f)-\sum_{r \in R_v} c_r \la \alpha_r,v \ra \frac{f-r(f)}{\alpha_r}-\sum_{r \in R \setminus R_v} c_r \la \alpha_r,v \ra \frac{f-r(f)}{\alpha_r}.$$ Now observe that $f-r(f) \in K \subseteq J$ for all $r$, and that if $\alpha_r g \in J$ for some $r \in R \setminus R_v$ and $g \in \CC[\hh]$ then $g \in J$. Thus each term in the last sum belongs to $J$. Finally, use the fact that $I$ is $H_c(W_v,U)$-stable.
\end{proof} 

Later on, we will apply this lemma to construct submodules in high rank starting with submodules in low rank. To use it, we have to start somewhere, and the next lemma states when the ideal of $0 \in \hh$ is $H_c(W,\hh)$-stable.
\begin{lemma}
The ideal $I \subseteq \CC[\hh]$ of  $0 \in \hh$ is $H_c(W,\hh)$-stable if and only if $c_E=1$ for all irreducible representations $E$ of $\CC W$ occurring in the degree one piece of $\CC[\hh]$. 
\end{lemma}
This is a special case of an observation of Etingof-Stoica \cite{EtSt}, and of the following theorem of Feigin \cite{Fei}:
\begin{theorem} \cite{Fei}
Let $S$ be a stratum of the reflection arrangement of $W$. Suppose $W_S=W_1 \times \cdots \times W_p$ is the product decomposition of $W_S$ into irreducible reflection subgroups, and write $\hh^*_i$ for the dual reflection representation of $W_i$. The ideal $I=I(W \cdot \overline{S})$ is $H_c$-stable if and only if $c_{\hh_i^*}=1$ for all $i$. 
\end{theorem}

\subsection{When the radical of $\CC[\hh]$ is a radical ideal} We now combine Feigin's theorem with our joint results with Juteau to obtain the promised characterization.

\begin{theorem}
Let $S$ be a stratum of the reflection arrangement of $W$. Suppose $W_S=W_1 \times \cdots \times W_p$ is the product decomposition of $W_S$ into irreducible reflection subgroups, and write $\hh^*_i$ for the dual reflection representation of $W_i$. Then $\mathrm{Rad}(\Delta_c(\triv))=I(W \cdot \overline{S})$ if and only if
\begin{enumerate}
\item[(a)] $c_{\hh^*_i}=1$ for all $i$, and
\item[(b)] The parameter $c$ does not belong to any positive hyperplane $H$ such that the quotient $s/s'$ of the principal Schur element $s$ for $W$ by the principal Schur element $s'$ for $W_S$ vanishes on the intersection of $H$ with the positive cone. 
\end{enumerate} Moreover, if the radical is a radical ideal, it must be of the form $I(W \cdot \overline{S})$ for some stratum $S$.
\end{theorem}
\begin{proof}
Suppose (a) and (b) hold. By condition (a), the radical contains the ideal $I(W \cdot \overline{S})$. When condition (b) holds, our result with Juteau implies that the support of $L_c(\triv)$ contains $W \cdot \overline{S}$, which gives the other containment. Conversely, assuming the radical is a radical ideal, by Bezrukavnikov-Etingof \cite{BeEt} it must be equal to some ideal $I(W \cdot \overline{S})$ for a single stratum $S$. By Feigin's theorem (a) holds, and by our theorem with Juteau, (b) holds.
\end{proof} We note that Corollary \ref{classical radical} is an immediate consequence, using the explicit form of the Schur elements for the group $G(\ell,1,n)$ from \cite{GrJu}.

\subsection{Cohen-Macaulayness of $L_c(\mathrm{triv})$}

We will use the following result, which is a slightly modification of Theorem 1.2 of \cite{EGL}. The proof is essentially the same as the homological duality proof in loc. cit., replacing the derived equivalence $\mathrm{RHom}(H,\cdot)$ between $\OO_c$ and $\OO_{c^*}$ there with the derived equivalence it induces between a block of $\OO_c$ and the corresponding block of $\OO_{c^*}$. 
\begin{theorem}  \cite{EGL}
Suppose $L$ is an irreducible object of $\OO_c$ which is of minimal support in its block. Then as a $\CC[\hh]$-module, $L$ is Cohen-Macaulay.
\end{theorem}

The \emph{positive cone} is the set of parameters $c$ satisfying $c_{H,\chi} \geq 0$ for all pairs $(H,\chi)$ as in \cite{GrJu}. For $c$ in the positive cone, the polynomial representation $\Delta_c(\mathrm{triv})$ is a projective object of $\OO_c$.
\begin{theorem}
Suppose the parameter $c$ belongs to the positive cone. Then the module $L_c(\mathrm{triv})$ is of minimal support in its block, and hence is Cohen-Macaulay.
\end{theorem}
\begin{proof}
Fix a point $p \in \hh$, and consider the unit-co-unit composition for the restriction-induction bi-adjunction $(\mathrm{Res}_p,\mathrm{Ind}_p)$, $$1 \to \mathrm{Ind}_p \mathrm{Res}_p \to 1.$$ This gives an element $z_p$ of the center of $\OO_c$, which must therefore act as a scalar on each block. Given an irreducible module $L$ of the rational Cherednik algebra, if $p \notin \mathrm{supp}(L)$ then this scalar must be zero on the block to which $L$ belongs. But for $L=L_c(\mathrm{triv})$, this scalar is the relative principal Schur element $|W:W_p|_{q_c}$. If this is zero then $p \notin \mathrm{supp}(L)$: in this case, since $c$ is positive $\Delta=\Delta_c(\mathrm{triv})$ is projective, and by Corollary 3.3 of \cite{GrJu} $p$ is not contained in the support of $L_c(\mathrm{triv})$. 
\end{proof}
Theorem \ref{CM theorem} follows by combining this with Corollary \ref{classical radical}, noting that the positive cone consists of those parameter $(c_0,d_0,\dots,d_{r-1})$ such that     $c_0 \geq 0$ and $d_i \geq d_{i+1}$ for $0 \leq i \leq \ell-2$. 

\section{Homology} \label{homology}

\subsection{Conventions for graded algebras and graded modules} We will work throughout with $\ZZ$-graded algebras and $\CC$-graded modules. 

\subsection{Betti numbers} Let $M$ be a $\CC$-graded $S=\CC[x_1,\dots,x_n]$-module and let $\CC$ be the one-dimensional $S$-module on which each $x_i$ acts as $0$. We define $F(M)=M \otimes_S \CC$, which is a graded vector space. This defines a functor from the category $S$-modg of finitely generated graded $S$-modules to the category Vectg of finite-dimensional graded vector spaces. It is right exact, and its left derived functor $LF$ is then a functor  $LF:D^b($S$\text{-modg}) \to D^b(\mathrm{Vectg})$ from the bounded derived category of graded $S$-modules to the bounded derived category of finite dimensional graded vector spaces. The Tor groups $\mathrm{Tor}_i(M,\CC)$ are the homology groups of the complex $LF(M)$; these are then graded vector spaces. The \emph{graded Betti numbers} of $M$ record the graded dimensions of these Tor groups. More precisely, the graded Betti numbers of $M$ are the dimensions
$$\beta_{i,j}=\mathrm{dim}(\mathrm{Tor}_i(M,\CC)_j),$$ where we write $\mathrm{Tor}_i(M,\CC)_j$ for the degree $j$ piece of the Tor group. 

In our applications, we will always work with modules $M$ that are also equipped with an action of a finite group $W$ of automorphisms of $S$. We package this extra structure as follows: let $S \rtimes W$ be the twisted group ring of $W$ over $S$: as a vector space, this is simply $S \otimes \CC W$, with multiplication given by
$$(s_1 \otimes w_1) \cdot (s_2 \otimes w_2)=s_1 w_1(s_2) \otimes w_1 w_2.$$ The injection $S \hookrightarrow S \rtimes W$ induces an embedding of the category $S \rtimes W$-modg of finitely generated graded $S \rtimes W$-modules into $S$-modg. We will write $F$ also for the functor $F$ induces from $S \rtimes W$-modg to the category $\CC W$-modg of finite-dimensional graded $\CC W$-modules. Thus for a finitely generated graded $S \rtimes W$-module $M$ the Tor groups $\mathrm{Tor}_i(M,\CC)$ are finite dimensional graded $\CC W$-modules.

\subsection{$\hh^*$-homology and Tor groups} Since every object of $\OO_c$ is a finitely-generated graded $S=\CC[\hh]$-module, we have an embedding $e:\OO_c \hookrightarrow S \rtimes W\mathrm{-grmod}$ of $\OO_c$ into the category of finitely generated graded $S \rtimes W$-modules. This embedding is exact and maps projective objects in $\OO_c$ to free $S$-modules, and it follows that the left derived functors of the composite $\CC \otimes e(\cdot)$ are the Tor groups $\mathrm{Tor}_i(e(\cdot),\CC)$. These left derived functors are denoted $H_i(\hh^*,M)$, and we refer to them as the \emph{$\hh^*$-homology} of $M$. Thus for $M \in \OO_c$ we have an equality
$$\mathrm{Tor}_i(M,\CC)=H_i(\hh^*,M) \quad \hbox{for all $i \in \mathbf{Z}$,}$$ where on the left-hand side we regard $M$ simply as an $S=\CC[\hh]$-module. Here and in the rest of the paper we will omit mention of the forgetful functor taking an object $M$ of $\OO_c$ to a graded $\CC[\hh] \rtimes W$-module.

\subsection{$\hh$-cohomology, and duality} Given a module $M$ in $\OO_c$, we define its \emph{$\hh$-cohomology} as follows: firstly,
$$H^0(\hh,M)=\{m \in M \ | \ ym=0 \quad \hbox{for all $y \in \hh$} \}.$$ Evidently this is a left-exact functor from $\OO_c$ to the category $\CC W$-modg of graded $\CC W$-modules, and abbreviating it by $G(M)=H^0(\hh,M)$ we write $RG$ for its right derived functor. We define $H^i(\hh,M)$ to be the $i$th cohomology group of the complex $RG(M)$. The duality $M \mapsto M^\vee$ on category $\OO_c$ relates the functors $RG$ and $LF$ as follows:

\begin{theorem} \label{duality theorem} There is an isomorphism of functors 
$$LF(M)^\vee \cong RG(M^\vee).$$
\end{theorem}
\begin{proof}
By definition, the canonical map
$$M \to (M^\vee)^\vee \quad m \mapsto [\phi \mapsto \overline{\phi(m)}]$$ is an isomorphism. By observing that the annihilator of $\hh^* M$ in $M^\vee$ is $H^0(\hh,M^\vee)$, we obtain an induced isomorphism of functors
$$H_0(\hh^*,M)=M / \hh^* M \cong (M^\vee)^\vee/H^0(\hh,M^\vee)^\perp=H^0(\hh,M^\vee)^\vee,$$ and hence an isomorphism of derived functors $LF(M) \cong RG(M^\vee)^\vee$. Taking duals produces the desired isomorphism $LF(M)^\vee \cong RG(M^\vee)$. 
\end{proof}

\subsection{$\hh$-cohomology and Ext groups} 

The functor $H^0(\hh,M)$ decomposes as a direct sum
$$H^0(\hh,M)=\bigoplus_{E \in \mathrm{Irr}(\CC W)} \mathrm{Hom}_{\CC W}(E,H^0(\hh,M)) \otimes E$$ implying
$$H^i(\hh,M)= \bigoplus_{E \in \mathrm{Irr}(\CC W)} \mathrm{Hom}_{\CC W}(E,H^i(\hh,M)) \otimes E$$There is an isomorphism of functors 
$$ \mathrm{Hom}_{\CC W}(E,H^0(\hh,M)) \cong \mathrm{Hom}_{\OO_c}(\Delta_c(E),M).$$ Taking right derived functors gives
$$\mathrm{Ext}^i_{\OO_c}(\Delta_c(E),M) \cong  \mathrm{Hom}_{\CC W}(E,H^i(\hh,M)).$$ 
Thus
$$H^i(\hh,M) \cong \bigoplus_{E \in \mathrm{Irr}(\CC W)} \mathrm{Ext}^i_{\OO_c}(\Delta_c(E),M) \otimes E$$ Now for each irreducible $\CC W$-module $F$, the irreducible object $L_c(F)$ of $\OO_c$ is self dual, $L_c(F)^\vee \cong L_c(F)$. Hence by Theorem \ref{duality theorem} we obtain an isomorphism of $\CC W$-modules
$$H_i(\hh^*,L_c(F)) \cong H_i(\hh^*,L_c(F))^\vee \cong H^i(\hh,L_c(F)) \cong \bigoplus_{E \in \mathrm{Irr}(\CC W)} \mathrm{Ext}^i_{\OO_c}(\Delta_c(E),L_c(F)) \otimes E.$$

\subsection{Tor and Ext} \label{Tor and Ext} Finally, putting together the previous isomorphisms, regarding $L_c(F)$ simply as an $S=\CC[\hh]$-module we have
$$\mathrm{Tor}_i(L_c(F),\CC) \cong H_i(\hh^*,L_c(F)) \cong  \bigoplus_{E \in \mathrm{Irr}(\CC W)} \mathrm{Ext}^i_{\OO_c}(\Delta_c(E),L_c(F)) \otimes E.$$ This shows that we may compute the Betti numbers of $L_c(F)$ in terms of Kazhdan-Lusztig multiplicities
$$m_{EF}^i=\mathrm{dim}_\CC(\mathrm{Ext}^i_{\OO_c}(\Delta_c(E),L_c(F))).$$ Presently we will see that in fact the graded Betti numbers may be computed via this formula as well.

\subsection{The Clifford algebra} Here we follow \cite{Ciu}. Let $V=\hh^* \oplus \hh$, and define a non-degenerate symmetric form $(\cdot,\cdot)$ on $V$ by taking $\hh^*$ and $\hh$ to be isotropic, and requiring $(x,y)=x(y)$ for $x \in \hh^*$ and $y \in \hh$. The \emph{Clifford algebra} $C(V)$ for this form is the algebra generated by $\hh$ and $\hh^*$ subject to the relations: $x^2=0=y^2$ for all $x \in \hh^*$ and all $y \in \hh$, and
$$yx+xy=( x,y)  \quad \hbox{for $x \in \hh^*$ and $y \in \hh$.}$$ The exterior algebras $\Lambda^\bullet \hh$ and $\Lambda^\bullet \hh^*$ are subalgebras of $C(V)$ and multiplication induces a vector space isomorphism $\Lambda^\bullet \hh^* \otimes \Lambda^\bullet \hh \rightarrow C(V)$. 

\subsection{The spin representation of $C(V)$} The \emph{spin representation} of $C(V)$ is the induced representation 
$$\Lambda^\bullet \hh^*=\mathrm{Ind}_{\Lambda^\bullet \hh}^{C(V)} \CC,$$ where $\CC$ is the one-dimensional $\Lambda^\bullet \hh$-module on which the elements $y \in \hh$ all act by $0$. As a $\Lambda^\bullet \hh^*$-module the spin representation is isomorphic to $\Lambda^\bullet \hh^*$. The spin representation carries a positive definite Hermitian contravariant form defined in the obvious fashion.

\subsection{Cohomology via Koszul complex} To compute $H_i(\hh^*,M)$ one introduces the Koszul complex $M \otimes \Lambda^i \hh^*$ with differential obtained as follows: let $y_1,\dots,y_n$ be a basis of $\hh$ with dual basis $x_1,\dots,x_n$ of $\hh^*$. Working in the tensor product algebra $\CC[\hh] \otimes C(V)$, define 
$$D_x=\sum x_i \otimes y_i \in \CC[\hh] \otimes C(V),$$ which is independent of the choice of basis. A straightforward computation shows $D_x^2=0$. Given any $\CC[\hh]$-module $M$, the operator $D_x$ then defines a differential on the complex $M \otimes \Lambda^\bullet \hh^*$, and the homology of this complex is precisely $H_i(\hh^*,M)$ (as we will explain in the next subsection). 

\subsection{Weyl and Clifford} One defines similarly 
$$D_y=\sum y_i \otimes x_i \in \CC[\hh^*] \otimes C(V).$$ Writing $D(\hh)$ for the Weyl algebra of polynomial coefficient differential operators on $\hh$, then identifying $y_i $ with $\frac{\partial}{\partial x_i}$ we may regard both $D_x$ and $D_y$ as elements of $D(\hh) \otimes C(V)$, and check that as such we have $D_x^2=0=D_y^2$ and
$$D_x D_y+D_y D_x=\sum_{i=1}^n x_i y_i \otimes 1+1 \otimes x_i y_i.$$ This identity shows that the complex $\CC[\hh] \otimes \Lambda^i \hh^*$ is exact except at $i=0$, where it has homology $\CC$,  the one-dimensional $\CC[\hh]$ module on which $\hh^*$ acts trivially.  Thus the Koszul complex for $\CC[\hh]$ is a free resolution of $\CC$, so the homology groups $H_i(\hh^*,M)$ may be computed by tensoring $M$ with this and taking homology. 

\subsection{The Dirac operator} The remainder of this section largely follows \cite{HuWo} and \cite{Ciu}. We now deform the previous construction, replacing $D(V)$ with $H_c$. Specifically, we will consider the tensor product algebra $H_c \otimes C(V)$. The \emph{Dirac operator} is the element $D \in H_c \otimes C(V)$ defined by
$$D=\sum_{i=1}^n (x_i \otimes y_i+y_i \otimes x_i)$$ for any choice of dual bases $y_1,\dots,y_n$ and $x_1,\dots,x_n$ as above. We have
$$D=D_x+D_y \quad \text{with} \quad D_x=\sum x_i \otimes y_i \quad \text{and} \quad D_y=\sum y_i \otimes x_i.$$ Since
$$D_x^2=\sum_{i,j} x_i x_j \otimes y_i y_j=\sum_{i < j} (x_i x_j-x_j x_i) \otimes y_i y_j=0$$ and similarly $D_y^2=0$, we obtain 
$$D^2=D_x D_y+D_y D_x.$$ Now \begin{align*}
D_x D_y+D_y D_x&=\sum_{i,j} x_i y_j \otimes y_i x_j+ \sum_{i,j} y_i x_j \otimes x_i y_j=\sum_{i,j} x_i y_j \otimes y_i x_j+y_j x_i \otimes x_j y_i \\
&=\sum_{i,j} x_i y_j \otimes y_i x_j+(x_i y_j + \delta_{ij}-\sum_{r \in R} c_r \la \alpha_r,y_j \ra \la x_i,\alpha_r^\vee \ra r) \otimes x_j y_i \\
&=\sum_{i,j} x_i y_j \otimes (y_i x_j+x_j y_i)+(\delta_{ij}-\sum_{r \in R} c_r \la \alpha_r,y_j \ra \la x_i,\alpha_r^\vee \ra r) \otimes x_j y_i \\
&=\sum_i (x_i y_i \otimes 1+ 1 \otimes x_i y_i)-\sum_{i,j,r} c_r r \otimes \la \alpha_r, y_j \ra x_j \la x_i,\alpha_r^\vee \ra y_i \\
&=\sum_i ( x_i y_i \otimes 1 + 1 \otimes x_i y_i) -\sum_{r \in R} c_r r \otimes \alpha_r \alpha_r^\vee \\
&=\mathrm{eu} \otimes 1 + \sum_i 1 \otimes x_i y_i -\sum_r ( c_r (1-r) \otimes 1 + c_r r \otimes \alpha_r \alpha_r^\vee) \\
&=\mathrm{eu} \otimes 1 + \sum_i 1 \otimes x_i y_i -\sum_r c_r (1 \otimes 1- r \otimes (1-\alpha_r \alpha_r^\vee)) \\
&=\mathrm{eu} \otimes 1 + \sum_i 1 \otimes x_i y_i -\sum_r c_r (1-r \otimes \tilde{r})
\end{align*} where for $r \in R$ we write $\tilde{r}=1-\alpha_r \alpha_r^\vee$. We have obtained:
\begin{equation} \label{Dx Dy}
D_x D_y+D_y D_x=\mathrm{eu} \otimes 1 + \sum_i 1 \otimes x_i y_i -\sum_r c_r (1-r \otimes \tilde{r})
\end{equation} as elements of $H_c \otimes C(V)$ and hence
\begin{equation} \label{D squared}
D^2=\mathrm{eu} \otimes 1 + \sum_i 1 \otimes x_i y_i -\sum_r c_r (1-r \otimes \tilde{r})
\end{equation} We remark that the action of $\tilde{r}$ on the $i$th graded piece $\Lambda^i \hh^*$ is the same as the action of $r \in W$.

\subsection{The grading on Tor and equivariant purity}

The Tor-groups of a $\CC[\hh]$-module $M$ may be calculated via the resolution $M \otimes \Lambda^i \hh^*$ with differential $D_x$. 

\begin{corollary} \label{purity}
Let $M$ be an object of $\OO_c$. If the $E$-isotypic component of the $j$th graded piece of the Tor group is non-zero 
$$\mathrm{Tor}_i(M,\CC)_{d,E} \neq 0$$ then we must have 
$$d=c_E.$$ 
\end{corollary} 
\begin{proof}
By equation \eqref{Dx Dy}, a for a cocycle $c \in M_j \otimes \Lambda^i \hh^*$  representing a non-zero homology class  of isotype $E$ and degree $d=i+j$ we have
$$D_x D_y(c)=(j+i-c_E) c \quad \implies \quad j+i-c_E=0.$$
\end{proof}

\subsection{Dirac cohomology}

If $M$ is a $H \otimes C(V)$-module then we define the \emph{Dirac cohomology} of $M$ by 
$$H_D(M)=\mathrm{ker}_M(D)/\mathrm{ker}_M(D) \cap \mathrm{im}_M(D).$$ When $M$ is an $H$-module we abuse notation somewhat by defining
$$H_D(M)=H_D(M \otimes \Lambda^\bullet \hh^*).$$ We hope this will not cause confusion.

The Dirac cohomology may be calculating within $\mathrm{ker}(D^2)$:  evidently 
$$H_D(M)=\mathrm{ker}_{\mathrm{ker}(D^2)}(D)/\mathrm{im}_{\mathrm{ker}(D^2)}(D).$$ That is, it is the cohomology of the operator $D$ restricted to $\mathrm{ker}(D^2)$.

Dirac cohomology is $\ZZ/2$-graded: the operator $D$ is of degree one with respect to the $\ZZ/ 2$-gradation given by
$$(M \otimes \Lambda^\bullet \hh^*)^\mathrm{even}=\bigoplus_{i \in 2 \ZZ} M \otimes \Lambda^i \hh^* \quad \text{and} \quad (M \otimes \Lambda^\bullet \hh^*)^\mathrm{odd}=\bigoplus_{i \in 2 \ZZ +1} M \otimes \Lambda^i \hh^*.$$ We write
$$H_D(M)=H]_D(M)^\mathrm{even} \oplus H_D(M)^\mathrm{odd}$$ for the corresponding decomposition of $H_D(M)$.

It follows from \eqref{D squared} that if $M$ is $\mathrm{eu}$-diagonalizable, then the $D^2$ action on $M \otimes \Lambda^\bullet \hh^*$ is given by the formula
$$D^2m=(a+i-c_E) m \quad \hbox{for $m \in (M_a \otimes \Lambda^i \hh^*)_E$,}$$ where for a $W$-module $N$ we write $N_E$ for the $E$-isotypic subspace. In particular the kernel of $D^2$ is the sum 
\begin{equation}
\mathrm{ker}(D^2)=\bigoplus_{c_E=a+i} (M_a \otimes \Lambda^i \hh^*)_E,
\end{equation} and is hence finite dimensional.

\subsection{Hodge decomposition for unitary modules} Suppose $M$ is a unitary representation of $H$. Then product of the form on $M$ and the contravariant form on the spin representation gives a form on $M \otimes \Lambda^\bullet \hh^*$ which is positive definite as well. 
\begin{lemma} Suppose $M$ is a unitary $H$-module on which $\mathrm{eu}$ acts locally finitely.
\begin{itemize}
\item[(a)] $M \otimes \Lambda^\bullet \hh^*$ is an orthogonal direct sum $M \otimes  \Lambda^\bullet \hh^*=\mathrm{ker}(D^2) \oplus \mathrm{im}(D^2)$.
\item[(b)]  $\mathrm{ker}(D_x) \cap \mathrm{ker}(D_y) = \mathrm{ker}(D) = \mathrm{ker}(D^2)$.
\item[(c)] $\mathrm{im}(D_x)+\mathrm{im}(D_y)=\mathrm{im}(D)=\mathrm{im}(D^2)$.
\item[(d)] $\ker(D_x)=\mathrm{ker}(D^2) \oplus \mathrm{im}(D_x)$ is an orthogonal direct sum, and likewise for $$\mathrm{ker}(D_y)=\mathrm{ker}(D^2) \oplus \mathrm{im}(D_y).$$
\end{itemize}
\end{lemma}
\begin{proof}
For (a) we observe that since $\mathrm{eu}$ is self-adjoint and locally finite, it is diagonalizable. It follows that 
$$\mathrm{ker}(D^2)=\bigoplus_{a+i=c_E} (M_a \otimes \Lambda^i \hh^*)_E \quad \text{and} \quad \mathrm{im}(D^2)=\bigoplus_{a+i \neq c_E} (M_a \otimes \Lambda^i \hh^*)_E,$$ proving (a).

It is clear that $\mathrm{ker}(D_x) \cap \mathrm{ker}(D_y) \subseteq \mathrm{ker}(D) \subseteq \mathrm{ker}(D^2)$. Conversely, if $D^2v=0$ then $D_x D_y v=-D_y D_x v$ implies  
$$-(D_x v,D_x v)=(-D_y D_x v,v)=(D_x D_y v, v)=(D_y v,D_y v)$$ so that since $( \cdot,\cdot)$ is positive definite we have $D_xv=0=D_yv$, proving (b).

If $v \in \mathrm{ker}(D_x)$ then $0=(D_x(v),w)=(v,D_y(w))$ implies that $v \in \hh$ is orthogonal to $\mathrm{im}(D_y)$, and likewise $\mathrm{ker}(D_y)$ is orthogonal to $\mathrm{im}(D_x)$. Evidently $\mathrm{im}(D^2) \subseteq \mathrm{im}(D) \subseteq \mathrm{im}(D_x)+\mathrm{im}(D_y)$. If $v \in \mathrm{im}(D_x)$ or $v \in \mathrm{im}(D_y)$ then by (a) and the above it is orthogonal to $\mathrm{ker}(D^2)$ an hence belongs to $\mathrm{im}(D^2)$, proving (c).

By (b) $\mathrm{ker}(D^2)+\mathrm{im}(D_x) \subseteq \mathrm{ker}(D_x)$. Moreover, (a), (b) and (c) together imply that $M \otimes \Lambda^\bullet \hh^*$ is the orthogonal direct sum
$$M \otimes  \Lambda^\bullet \hh^*=\mathrm{ker}(D^2) \oplus \mathrm{im}(D_x) \oplus \mathrm{im}(D_y).$$ If $v \in \mathrm{ker}(D_x)$ then since $v$ is orthogonal to $\mathrm{im}(D_y)$ it belongs to $\mathrm{im}(D_x)+\mathrm{ker}(D^2)$. The proof for $D_y$ is the same.
\end{proof}
The next two results summarize what we have proved.
\begin{corollary} \label{unitary cohomology}
If $L=L_c(E)$ is a unitary $H_c$-module, then $$H_\bullet(\hh^*,L) \cong H_{D_x}(L) \cong H_D(L)=\mathrm{ker}(D^2) \cong H_{D_y}(L).$$ Moreover, we have
$$H_D(L)=\bigoplus_{c_F=a+i} (L_a \otimes \Lambda^i \hh^*)_F$$ and 
$$a+i \geq c_F \quad \hbox{for all $F$ such that $(L_a \otimes S^i)_F \neq 0$.}$$
\end{corollary}

\begin{corollary} Suppose $L$ is a unitary $H_c$-module. Then as graded $\CC W$-modules
$$H_i(\hh^*,L) \cong \bigoplus_{c_F=a+i} (L_a \otimes \Lambda^i \hh^*)_F$$ and
$$\mathrm{dim} \left( \mathrm{Ext}^i (\Delta_c(F),L) \right)=\mathrm{dim} \left( \mathrm{Hom}_{\CC W} (F,H_i(\hh^*,L) \right).$$
\end{corollary}

This completes the proof of Theorem \ref{Hodge decomposition}.

\section{The polynomial representation of the cyclotomic rational Cherednik algebra} \label{cyclotomic}

\subsection{} Here we specialize to the case of the cyclotomic reflection groups. We refer the reader especially to \cite{Gri} for the facts we will need about intertwining operators.

\subsection{The monomial groups}  For the remainder of the paper we will focus on one particular class of reflection groups. The group $G(\ell,1,n)$ consists of all $n$ by $n$ matrices with exactly one non-zero entry in each row and each column, and such that the non-zero entries are $\ell$th roots of $1$. We fix $\zeta$, a primitive $\ell$th root of $1$, and will write $\zeta_i$ for the diagonal matrix with $1$'s on the diagonal except in position $i$, where $\zeta$ appears. Write $s_{ij}$ for the transposition matrix interchanging the $i$th and $j$th basis vectors. The group $G(\ell,1,n)$ contains $\ell$ conjugacy classes of reflections: the class containing $s_{12}$, and for each $1 \leq k \leq \ell-1$, the class containing $\zeta_1^k$.

\subsection{Parabolic subgroups of $G(\ell,1,n)$} Let $v \in \hh=\CC^n$ have $m$ coordinates equal to zero. Some of the non-zero coordinates may be equal to one another up to multiplication by $\ell$th roots of one; by multiplication by the appropriate diagonal matrices in $G(\ell,1,n)$ we assume that the first $m$ coordinates are zero, the next $n_1$ are equal, the next $n_2$ are equal, and so on, and that there are no other equalities between any coordinates. The stabilizer of $\hh$ is then the subgroup
$$G(\ell,1,m) \times S_{n_1} \times S_{n_2} \times \cdots \times S_{n_k} \subseteq G(\ell,1,n) $$ with $m+n_1+\cdots+n_k=n$. This shows that up to conjugacy, every proper parabolic subgroup of $G(\ell,1,n)$ is of the form $G(\ell,1,m) \times S_{n_1} \times S_{n_2} \times \cdots \times S_{n_k}$ with $0 \leq m \leq n-1$ and $(n_1 \geq n_2 \geq \cdots)$ a partition of $n-m$. There are inclusions between these governed by the merge order on partitions. In particular, the maximal parabolic subgroups, up to conjugacy, are 
$$G(\ell,1,m) \times S_{n-m} \quad \hbox{for $0 \leq m \leq n-1$.}$$

\subsection{The cyclotomic rational Cherednik algebra} 
We will write $c=(c_0,d_0,d_1,\dots,d_{\ell-1})$ for the deformation parameter, with the conventions that $$d_0+d_1+\cdots+d_{\ell-1}=0 \quad \text{and} \quad d_j=d_{j+\ell}$$ for all $j \in \ZZ$ (this determines $d_j$ for all $j$ once $d_1,\dots,d_{\ell-1}$ are fixed). We note that if $\ell=2$ then setting $d=d_0$ gives the usual type B parameter. 

The cyclotomic rational Cherednik algebra $H_c$ is generated by $\CC[x_1,\dots,x_n]$, $\CC[y_1,\dots,y_n]$, and the group $W=G(\ell,1,n)$ subject to the relations $w f w^{-1}=w \cdot f$ for $w \in W$ and $f \in \CC[x_1,\dots,x_n]$ or $f \in \CC[y_1,\dots,y_n]$, 
 \begin{equation}
 y_i x_i=x_i y_i+1-c_0 \sum_{\substack{ 1 \leq j \neq i \leq n \\ 0 \leq r \leq \ell-1}} \zeta_i^r s_{ij} \zeta_i^{-r}-\sum_{r=0}^{\ell-1} (d_r-d_{r-1}) e_{ir}
 \end{equation} for $1 \leq i \leq n$, where $$e_{ir}=\frac{1}{\ell} \sum_{t=0}^{\ell-1} \zeta^{-tr} \zeta_i^t$$ and
\begin{equation}
y_i x_j=x_j y_i+ c_0 \sum_{r=0}^{\ell-1} \zeta^{-r} \zeta_i^r s_{ij} \zeta_i ^{-r}
\end{equation} for $1 \leq i \neq j \leq n$.  

\subsection{Basic submodules of type $E$} 

The first type of submodule we will consider occurs for $c_0=j/k$, where $j$ and $k$ are positive coprime integers with $2 \leq k \leq n$. Let $I$ be the (unique) non-trivial submodule of the polynomial representation of the type $S_k$ rational Cherednik algebra at parameter $c_0=j/k$; the zero set of $I$ is the origin. Applying Lemma \ref{submodule construction} to a product of $s$ copies of $S_k$ embedded in $G(\ell,1,n)$ gives a submodule $E_s$ of $\CC[\hh]$, with $E_s \subseteq E_{s+1}$. Though we will not need or prove this here, these exhaust the non-trivial submodules of $\CC[\hh]$ for $c$ Weil generic subject to $c_0=j/k$. 

With $c_0=j/k$ fixed as above, we will continute to write $E_1 \subseteq \cdots \subseteq E_{\lfloor n/k \rfloor}$ for the chain of non-trivial submodules constructed above. Examining the construction shows that the zero set of $E_s$ consists of those points $x=(x_1,\dots,x_n)$ that possess $s$ disjoint blocks of $k$ coordinates each such that the $\ell$th powers of the coordinates in a given block are equal to one another. It is thus a generalization of the $k$-equals arrangement (and this explains our choice of the notation $E$ for this type of submodule). 

\subsection{Non-symmetric Jack polynomials} Following \cite{DuOp} we define
$$z_i=y_i x_i+c_0 \phi_i,$$ where $$\phi_i=\sum_{\substack{1 \leq j < i \\ 0 \leq k \leq \ell-1}} \zeta_i^k s_{ij} \zeta_i^{-k}.$$ The subalgebra $\ttt$ of $H_c$ generated by $z_1,\dots,z_n,\zeta_1,\dots,\zeta_n$ is commutative, and acts in an upper-triangular fashion on $\CC[\hh]=\CC[x_1,\dots,x_n]$ with respect to a certain partial ordering of the basis of monomials. Treating $c_0$ as a formal variable for the moment, the eigenfunctions are polynomials whose coefficients are rational functions of $c_0$, and we write $f_\mu \in \CC(c_0)[x_1,\dots,x_n]$ for the eigenfunction with leading term equal to $x^\mu=x_1^{\mu_1} \cdots x_n^{\mu_n}$. We refer to $f_\mu$ as the \emph{non-symmetric Jack polynomial of type $G(\ell,1,n)$}. Dunkl and Opdam (Cor. 3.13 of \cite{DuOp}) have observed the following relationship between these polynomials for $\ell=1$ and for $\ell$ arbitrary: writing $g_\mu$ for the $\ell=1$ version, we have
\begin{equation} \label{ell power}
f_{\ell \mu}=g_\mu(x_1^\ell,\dots,x_n^\ell) \quad \hbox{for all $\mu \in \ZZ_{\geq 0}^n$.} \end{equation}

We will first describe the submodule structure of $\CC[x_1,\dots,x_n]$ in case $\ell=1$ (so $W=S_n$ is the symmetric group). We need a bit of notation. Fix coprime positive integers $j$ and $k$ with $k \geq 2$. For each integer $s$ with $1 \leq s \leq \lfloor n/k \rfloor$, we divide $n-(sk-1)$ by $k-1$ to obtain a quotient $q_s$ and remainder $r_s$ defined by
$$n-(sk-1)=q_s(k-1)+r_s \quad \hbox{with $q_s, r_s \in \ZZ_{\geq 0}$ and $r_s < k-1$.}$$ We define $\mu^s \in \ZZ_{\geq 0}^n$ by
$$\mu^{s}=((j(s+q_s))^{r_s},(j(s+q_s-1))^{k-1}, (j(s+q_s-2))^{k-1},\dots, (j s)^{k-1},0^{sk-1}).$$ We also define a partition $\tau^s$ by
$$\tau^{s}=(sk-1,(k-1)^{q_s},r_s).$$

Now combining work of Dunkl \cite{Dun2} (see also \cite{Dun3}) and Etingof-Stoica \cite{EtSt} as in \cite{BGS} we have the following description of the submodule structure of $\CC[x_1,\dots,x_n]$ when $\ell=1$:
\begin{theorem}
The non-symmetric Jack polynomial $g_{\mu^s}$ has no pole at $c_0=j/k$, the non-trivial submodules of $\CC[x_1,\dots,x_n]$ at $c_0=j/k$ are precisely $E_s$ for $1 \leq s \leq \lfloor n/k \rfloor$, we have $E_s=H_c \cdot g_{\mu_s}$ and $E_s \subseteq E_{s+1}$ for $1 \leq s \leq \lfloor n/k \rfloor -1$. Moreover, the lowest weight space of $E_s$ is $S^{\tau^s}$, which is generated as an $S_n$-module by $g_{\mu^s}$. 
\end{theorem}

We observe that by \eqref{ell power} the $\CC G(\ell,1,n)$-submodule generated by $f_{\ell \mu}=g_{\mu^s}(x_1^\ell,\dots,x_n^\ell)$ is isomorphic to the representation indexed by the $\ell$-partition $(\tau^s,\emptyset,\dots,\emptyset)$. The next lemma is a consequence of this fact, the preceding theorem, Theorem 3.3  from \cite{BGS} and the formula $z_i=y_i x_i+c_0 \phi_i$.

\begin{lemma} \label{E1 basis}
Suppose $c$ is a parameter with $c_0=j/k$. We have $E_s=H_c \cdot f_{\ell \mu^s}$ for all $1 \leq s \leq \lfloor n/k \rfloor$, and the lowest weight space of $E_s$ is $\CC G(\ell,1,n) \cdot f_{\ell \mu^s}$, which is isomorphic to the $G(\ell,1,n)$-module indexed by the multi-partition $(\tau^s,\emptyset, \dots,\emptyset)$. In particular, for any choice of parameter $c$, the submodule $E_1$ has the vector space basis
$$E_1=\CC \{ f_\mu \ | \ \hbox{$\mu$ is $(j \ell, k)$-admissible} \}$$
\end{lemma}

\subsection{Lattices} In this subsection $R$ is a principal ideal domain with field of fractions $F$, $L$ is a free $R$-module of finite rank, and $V=F \otimes_R L$. We regard $L$ as a subset of $\hh$ via the inclusion $\ell \mapsto 1 \otimes \ell$ for $\ell \in L$.

\begin{lemma} \label{lattice lemma}
Let $W \subseteq V$ be a subspace and define $M=W \cap L$. Then
\begin{enumerate}
\item[(a)] The canonical map $F \otimes_R M \to W$ given by $f \otimes m \mapsto fm$ is an isomorphism.

\item[(b)] $M$ is a free $R$-module of rank equal to the dimension of $W$.
\end{enumerate}
\end{lemma}
\begin{proof}
For each $w \in W$ there is some $r \in R$ such that $rw \in L$. Thus the map in (a) is surjective. It is also injective since $F$ is a flat $R$-module, proving (a) (this part uses only that $R$ is an integral domain). For (b), we observe that $M \subseteq L$ is free since submodules of free $R$-modules are free, and by (a) its rank is the dimension of $W$. 
\end{proof}

\subsection{Basic submodules of type $Z$}

We will use the previous lemma together with the theory of non-symmetric Jack polynomials to describe a collection of submodules that we conjecture generate the lattice of submodules of $\Delta_c(\mathrm{triv})$.  Let $R=\CC[c]$ be the polynomial ring in one variable $c$ and let $F=\CC(c)$ be its field of fractions. We fix $d_0,\dots,d_{\ell-1} \in F$ with $d_0+\cdots+d_{\ell-1}=0$. We will write $H_F=F \otimes H(W,V)$, $\Delta_F(\mathrm{triv})$, etc... for the base change to $F$ corresponding to the choice of parameters $d_0,\dots,d_{\ell-1}$ and $c_0=c$ generic. Similarly, $H_R$, $\Delta_R(\mathrm{triv})$, etc... denote the base change to $R$.

We now define certain $R$-submodules of $\Delta_R(\mathrm{triv})$, which will turn out to be $H_R$-stable along certain hyperplanes in the parameter space. Given integers $k$ and $m$ with $k>0$ and $0 \leq m \leq n-1$ we put
$$W_{k,m}=F \{f_\mu \ | \ \mu^-_{m+1} \geq k \}$$ and 
$$M_{k,m}=W_{k,m} \cap \Delta_R(\mathrm{triv}).$$ By Lemma \ref{lattice lemma} applied to each graded piece of the standard module, $M_{k,m}$ is a free graded $R$-module of graded rank equal to the graded dimension of $W_{k,m}$.  

Suppose now that the parameters satisfy
$$d_0-d_{-k}+\ell m c_0=k  \quad \hbox{and $k \neq 0$ mod $\ell$.}$$ In this case \cite{Gri} shows $W_{k,m}$ is an $H_F$-stable subspace and hence $M_{k,m}$ is an $H_R$-submodule of $\Delta_R(\mathrm{triv})$. Given a further specialization of $c_0$ to a complex number, we will write $Z_{k,m}=\CC \otimes M_{k,m}$ for the specialization of $M_{k,m}$, which is an $H_c$-submodule of $\CC[\hh]$ of graded dimension equal to the graded dimension of $W_{k,m}$. Suppose that the parameters $c$ are (Weil) generic subject to the equation
$$d_0-d_{-k}+\ell mc_0=k \quad \hbox{and $k \neq 0$ mod $\ell$.}$$ In this case, by Theorem  7.5 of \cite{Gri2} , there is a unique non-trivial submodule of $\CC[\hh]$, and it follows that $Z_{k,m}$ is this unique non-trivial submodule. In particular, when $m=n-1$ the quotient of $\CC[\hh]$ by $Z_{k,n-1}$ is finite dimensional and hence the zero set of $Z_{k,n-1}$ is just the origin $0 \in \hh$. 

On the other hand, using Lemma \ref{submodule construction} we can construct a non-trivial submodule of $\CC[\hh]$ from the non-trivial submodule of the polynomial representation of $H_c(G(\ell,1,m+1)$. It follows that the zero set of $Z_{k,m}$ is the set of points $x \in \CC^n$ with at least $m+1$ zeros amongst the coordinates $x_1,\dots,x_n$ (this is our reason for using the letter $Z$).

\subsection{The submodule structure of $E_1$}

In Lemma \ref{E1 basis} we observed that the submodule $E_1$ is always $\ttt$-diagonalizable. Suppose that $$d_0-d_{-p}+\ell m c_0=p$$ for integers $p$ and $m$ with $p>0$, $m \geq 0$, and $p \neq 0$ mod $\ell$. Then the vector space
$$Z_{p,m}'=\CC \{f_\mu \ | \ \hbox{ $\mu$ is $(j \ell,k)$-admissible and $\mu_{m+1}^- \geq p$} \}$$ is an $H_c$-submodule of $E_1$. By applying exactly the same argument as in Theorem 7.5 from \cite{Gri2} it follows that its submodule lattice is generated by the submodules $Z_{p,m}'$ where $p>0$ and $m \geq 0$ are integers  such that the equation
$$d_0+d_{-p}+\ell m c_0=p$$ holds. In particular:
\begin{theorem} \label{socle thm}
Suppose $c_0=j/k$ for positive coprime integers $k$ and $j$ with $2 \leq k \leq n$. The socle $S$ of $\CC[\hh]$ has vector space basis
$$S=\CC\{f_\mu \ \vert \ \mu_{m+1}^- \geq p  \quad \hbox{if $d_0-d_{-p}+\ell m c_0=p$, $p \neq 0$ mod $\ell$, and $\mu$ is $(\ell j,k)$-admissible} \}.$$
\end{theorem}

Next we prove the following more precise version of Theorem \ref{basis thm}:

\begin{theorem} \label{radical socle basis}
Suppose $\ell \geq 2$ and $c$ satisfies the conditions from part (d) of Theorem \ref{radical socle}, so that in particular $c_0=1/k$ and $d_0-d_{\ell-1}+\ell m c_0=1$ with $m$ minimal. Then the socle of $\CC[\hh]$ is $E_1 \cap Z_{1,m}$, the ideal of the union $X_{k,\ell,n} \cup Y_{m+1,n}$, with vector space basis
$$E_1 \cap Z_{1,m}=\CC \{f_\mu \ \vert \ \mu_{m+1}^- \geq 1 \ \hbox{and $\mu$ is $(\ell,k)$-admissible} \}.$$ If moreover the inequalities 
$$d_0-d_{-p}+\ell m' c_0<p$$ hold for all pairs $(m',p)$ as in (d) of Theorem \ref{radical socle} with $0 < p < \ell$ then the socle is unitary.
\end{theorem}
\begin{proof}
For a generic choice of parameter $c$ satisfying the equation $c_0=j/k$ and $d_0-d_{-p}+ \ell m c_0=p$, we have also
$$d_0-d_{-p}+\ell (m+xk) c_0=p+x j \ell$$ for all $x \in \ZZ$, but no other equations relevant for the submodule structure of $E_1$ hold. Thus for such a choice of parameter, the lattice of submodules of $E_1$ is generated by those of the form $Z_{p+xj,m+xk}'$. Notice that if $\mu \in Z_{p,m}'$ then $\mu_{m+1}^- \geq p$ and $\mu$ is $(j \ell,k)$-admissible, implying $\mu_{m+1+k}^- \geq \mu_{m+k}^- \geq \mu_{m+1}^-+\ell j \geq p+\ell j$. Hence we have $Z_{p,m}' \subseteq Z_{p+j\ell,m+k}'$. We may and will assume $m \geq 0$ is minimal among all such equations that hold with positive $p$. Thus the only non-trivial submodules of $E_1$ are among the submodules in the chain
$$Z_{p,m}' \subseteq Z_{p+j \ell, m+k}' \subseteq Z_{p+2j \ell, m+2k}' \subseteq \cdots.$$

Suppose now that $j=p=1$ and $0 \leq m \leq k-2$ in the above. $Z_{1,m}$ is the ideal of the set $Y_{m+1,n}$ of points with at least $m+1$ zeros amongst their coordinates and $E_1$ is the ideal of the set $X_{k,\ell,n}$ of points such that some $k$ coordinates' $\ell$th powers are equal. The intersection $E_1 \cap Z_{1,m}$ is strictly smaller than $E_1$ (here we use $m \leq k-2$) and not zero, and hence must be equal to one of the $Z_{1+x \ell,m+xk}'$ for $x \in \ZZ_{\geq 0}$. We claim $x=0$. Otherwise, there is a non-symmetric Jack polynomial $f_{\ell \mu} \in E_1 \cap Z_{1,m}$ for some partition $\mu$ such that the last $k-1$ coordinates of $\mu$ are zero. But the ideal $Z_{1,m}$ is the monomial ideal generated by the squarefree monomials of the form $x^\nu$ with $\nu$ consisting of $n-m$ entries $1$ and the rest $0$. Thus $Z_{1,m}$ is equal to the span of those monomials $x^\nu$ in which there are at least $n-m$ positive entries. In particular, the leading term $x^{\ell \mu}$ of $f_{\ell \mu}$ is not in $Z_{1,m}$ (since $\mu$ has only $n-(k-1)<n-m$ positive entries), contradicting $f_{\ell \mu}  \in Z_{1,m}$. Thus $Z_{1,m}'=Z_{1,m} \cap E_1$ for generic parameters $c$ such that $c_0=1/k$ and $d_0-d_{\ell-1}+\ell m c_0=1$, and hence for all such choices of parameters.

Now suppose that the parameter $c$ satisfies the conditions from (d) of Theorem \ref{radical socle}. Then arguing as in \cite{Gri2} shows that the submodule $Z_{1,m}'$ is simple and hence equal to the socle. This proves the first assertion of the theorem, and the assertion about unitarity is proved along the same lines.
\end{proof}

With this description of $E_1 \cap Z_{1,m}$ in hand, we will deduce a more precise version of Theorems \ref{generation thm} and \ref{basis thm} for $\ell > 1$ or $m=0$. We write $\lambda$ for the $\ell$-partition $\lambda=(\lambda^0,\emptyset,\dots,\emptyset,\lambda^1)$ with $\lambda^0$ and $\lambda^1$ as in the introduction. There is a unique lowest degree occurrence of $S^\lambda$ in $\CC[x_1,\dots,x_n]$, which is spanned by the Garnir polynomials $f_T$. 
\begin{theorem}
Suppose $\ell \geq 2$. The ideal $E_1 \cap Z_{1,m}$ is generated by its lowest degree homogeneous piece, which is isomorphic to $S^\lambda$ as a $G(\ell,1,n)$-module and is the unique lowest-degree occurrence of $S^\lambda$ in $\CC[\hh]$. This lowest-degree occurrence of $S^\lambda$ has a vector space basis consisting of the (generalized) Garnir polynomials $f_T$ as $T$ ranges over all standard Young tableaux of shape $\lambda$.
\end{theorem}
\begin{proof}
We first observe that (for \emph{any} group $W$) the socle, being a simple module for $H_c$, is generated as a $\CC[\hh]$-module by its lowest degree homogeneous piece. Next, examining the $\ttt$-eigenvalues of the $f_\mu$'s in this lowest degree piece, the formula $z_i=y_i x_i+c_0 \phi_i$ allows one to compute the $\phi_i$-eigenvalues of the $f_\mu$'s, and a direct calculation then shows that this lowest degree homogeneous piece is isomorphic to $S^\lambda$ as a $\CC G(\ell,1,n)$-module. By degree considerations this is the unique lowest occurrence of $S^\lambda$ in $\CC[\hh]$, and it therefore has vector space basis consisting of the Garnir polyomials.
\end{proof}

Finally we complete the proof of Theorems \ref{generation thm} and \ref{basis thm} by deducing the case $\ell=1$ from the case $\ell=2$ proved above. We recall that in this paper we use the notation $f_\mu$ for the non-symmetric Jack polynomial for the group $G(\ell,1,n)$ and $g_\mu$ for the classical non-symmetric Jack polynomial, related by the formula
\begin{equation} \label{power relation}
f_{\ell \mu}=g_\mu(x_1^\ell,\dots,x_n^\ell). \end{equation}
\begin{theorem}
Let $I$ be the ideal of the set $X_{k,1,n} \cup Y_{m+1,n}$. As a vector space $I$ has basis
$$I=\CC \{g_\mu \ \vert \ \mu_{m+1}^- \geq 1 \ \hbox{and $\mu$ is $(1,k)$-admissible} \}.$$ Moreover $I$ is generated by its lowest degree homogeneous piece, which is spanned by the polynomials $f_T$ where $T$ ranges over all standard Young tableaux on $\lambda$.
\end{theorem}
\begin{proof}
To distinguish between the $S_n$ situation and the $G(\ell,1,n)$ situation we write $\CC[x_1,\dots,x_n]$ for the polynomial representation of the $G(\ell,1,n)$ Cherednik algebra, and $\CC[y_1,\dots,y_n]$ for the polynomial representation of the $S_n$ Cherednik algebra. We consider the map $s:\CC^n \to \CC^n$ that squares the coordinates, defined by $$s(a_1,\dots,a_n)=(a_1^2,\dots,a_n^2).$$ We have $s(X_{k,1,n} \cup Y_{m+1,n})=X_{k,2,n} \cup Y_{m+1,n}$. The induced map on coordinate rings $$s^*:C[y_1,\dots,y_n] \to \CC[x_1,\dots,x_n] \quad \text{with} \quad y_i \mapsto x_i^2 $$ is injective, and $s^*(I)=E_1 \cap Z_{1,m} \cap \CC[x_1^2,\dots,x_n^2]$. The first assertion of the theorem now follows from the relationship \eqref{power relation} between the classical Jack polynomials and the $G(2,1,n)$ version. The fact that $I$ is generated by its lowest degree homogeneous piece may be verified as follows: use the intertwining operators as in \cite{Gri} to see that the ideal
$$\CC \{f_\mu \ | \ \mu_{m+1}^- \geq 1 \ \hbox{and $\mu$ is $(2,k)$-admissible} \} \subseteq \CC[x_1,\dots,x_n]$$ is generated by its lowest degree piece, which then implies the same fact for $I$ by using the $(\mathbf{Z}/2)^n$ grading. Finally, the $\ell=2$ version of the polynomials $f_T$ generate the $\ell=2$ version of $I$, and again using the $(\mathbf{Z}/2)^n$-grading shows that the same thing is true for the $\ell=1$ version (though these are not linearly independent in general).
\end{proof}

We mention that for the type $G(2,1,n)$ Cherednik algebra, Feigin and Shramov have proved that the ideal generated by the $f_T$'s is unitary (and hence equal to the socle: it is semi-simple by unitarity and hence simple) provided $c_0=1/k$ and $d+mc=1/2$ for integers $k$ and $m$ with $2 \leq k \leq n$ and $0 \leq m \leq k-2$. 

\begin{proof} (Of  Theorem \ref{radical socle}) Now we will complete the proof of Theorem \ref{radical socle}. Assuming that the conditions from (d) hold, Theorem \ref{radical socle basis} proves that the socle is the ideal of $X_{k,n} \cup Y_{m+1,n}$. The cases (a), (b), and (c) are similar but easier. We suppose the socle $S$ is a radical ideal. If its zero set $V$ is all of $\CC^n$, then we are in case (a) of the theorem, and \cite{DuOp} implies that the conditions on the parameter stated there hold. Otherwise, $V$ must contain some stratum, with stabilizer group that we may assume is $G(\ell,1,m) \times S_{n_1} \times \cdots \times S_{n_r}$ for certain integers $m$ and $n_1,\dots,n_r$ with $m+n_1+\cdots+n_r=n$. 

Case 1.  We have $c_0=j/k$ for some coprime positive integers $j$ and to $k \geq 2$. By Theorem \ref{socle thm} the socle has basis
$$S=\CC\{f_\mu \ \vert \ \mu_{m+1}^- \geq p  \quad \hbox{if $d_0-d_{-p}+\ell m c_0=p$, and $\mu$ is $(\ell j,k)$-admissible} \}.$$ If the socle is the ideal of $X_{k,\ell,n}$ then this implies the parameters are as in case (b) of Theorem \ref{radical socle}. Otherwise the socle must be smaller, or in other words, its zero set must be larger than $X_{k,\ell,n}$. Hence its zero set must be $X_{k,\ell,n} \cup Y_{m+1,n}$ for some $0 \leq m \leq k-2$, and by Theorems \ref{socle thm} and \ref{radical socle basis} the parameters must be as in (d) of Theorem \ref{radical socle}.

Case 2. $c_0$ is not of the form $c_0=j/k$ for coprime positive integers $j$ and $k\geq 2$. Then by \cite{BeEt} the zero set $V$ of the socle is $V=Y_{m+1,n}$ for some $0 \leq m \leq n-1$, and the socle is therefore the ideal of $Y_{m+1,n}$. By Feigin's theorem we have $d_0-d_{\ell -1}+\ell m c_0=1$ and $m$ is minimal among all $0 \leq m' \leq n-1$ such that there exists a positive integer $p$ not divisible by $\ell$ with 
$$d_0-d_{-p}+\ell m' c_0=p.$$ Suppose therefore that such an equation
$$d_0-d_{-p}+\ell m' c_0=p$$ holds for some $m' \geq m$ and $p>1$. If $m'=m$ then $Z_{p,m} \subsetneq Z_{1,m}$ is $H_c$-stable and hence contains the socle, contradiction. Otherwise $m'>m$. In this case the equations
$$d_0-d_{\ell-1}+\ell m c_0=1 \quad \text{and} \quad d_0-d_{-p}+\ell m' c_0=p$$ both hold. For generic choices of $c_0$, this implies that the socle is strictly smaller than the ideal of $Y_{m+1,n}$, hence it is so for all choices of $c_0$, contradiction. The parameters must therefore be in case (c) of Theorem \ref{radical socle}.  \end{proof}
 
 Finally, we state two conjectures and one question on the submodule structure of the polynomial representation. 

\begin{conjecture}
With the notation as above, we have $Z_{p,m}'=Z_{p,m}\cap  E_1$. 
\end{conjecture}

This would be implied by the second conjecture, which is much stronger:

\begin{conjecture}[EZ submodule conjecture]
Suppose $c_0=j/k$ with $k$ and $j$ relatively prime positive integers, and $2 \leq k \leq n$. The lattice of submodules of $\CC[\hh]$ is generated by the submodules $E_1 \subseteq \cdots \subseteq E_{\lfloor n/k \rfloor}$ together with those $Z_{p,m}$ such that 
$$d_0-d_{-p}+\ell m c_0=p.$$
\end{conjecture}

The evidence for this conjecture is as follows: it is true for parameters $c_0$ that are not of the form $c_0=j/k$ for positive coprime integers $j$ and $k$ with $2 \leq k \leq n-1$ and it is true for generic parameters $c$ subject to the condition $c_0=j/k$. We hope that the submodule structure is then determined by this behavior in codimension $1$, as is the case in other problems of a Lie-theoretic flavor. Moreover, the conjecture predicts the correct support of $L_c(\mathrm{triv})$, as determined in \cite{GrJu}. The evidence against the conjecture: it would imply that there are only finitely many submodules of the polynomial representation, but it is easy to find examples of standard modules in category $\OO_c$ with infinitely many submodules (obviously we do not have an example of this behavior for the polynomial representation of type $G(\ell,1,n)$).

If the conjecture is true, it gives hope for solving a problem posed by Dunkl, de Jeu, and Opdam \cite{DJO}: determine the singular polynomials for the group $G(\ell,1,n)$. Each irreducible space of singular polynomials arises as the lowest weight of some submodule (thought not every submodule is generated by singular polynomials, and different submodules may have the same lowest degree subspace). 

\begin{question}
For which pairs $(W,c)$ is the submodule lattice of the polynomial representation of $H_c(W,\hh)$ finite?
\end{question}
\def\cprime{$'$} \def\cprime{$'$}

\end{document}